\documentclass[11pt, reqno]{amsart}
\makeatletter
\def\LaTeX{\leavevmode L\raise.42ex
    \hbox{\kern-.3em\size{\sfsize}{0pt}\selectfont A}\kern-.15em\TeX}
\makeatother

\newcommand{\BibTeX}{{\rm B\kern-.05em{\sc
          i\kern-.025emb}\kern-.08em\TeX}}

\makeatletter
\label{e:dispaa}
\label{e:dispau}
\label{e:dispav}
\label{e:dispaw}
\label{e:dispax}
\makeatother

\usepackage[english, french]{babel}
\usepackage{amsmath}
\usepackage{mathrsfs}    
\usepackage{amssymb}
\usepackage{amsthm}
\usepackage{graphicx,color}
\usepackage{hyperref}
\usepackage{verbatim} 

\usepackage[applemac]{inputenc}

\newtheorem{thm}{Theorem}[section]
\newtheorem{prop}[thm]{Proposition}
\newtheorem{defin}[thm]{Definition}

\newtheorem{lemma}[thm]{Lemma}

\newtheoremstyle{definition2}{\topsep}{\topsep}%
     {}
     {}
     {\bfseries}
     {.}
     {.5em}
     {\thmnumber{(#2)}\thmname{ #1}\thmnote{ #3}}

\theoremstyle{definition}
\newtheorem{rem}[thm]{Remark}

\def\ep{\varepsilon}
\def\vphi{\varphi}

\def\Feh{\mathcal{F}}

\def\Kah{\mathcal{K}}

\def\Mah{\mathcal{M}}
\def\Meh{\mathcal{M}}

\def\Peh{\mathcal{P}}

\def\N{\mathbb{N}}
\def\Z{\mathbb{Z}}

\def\R{\mathbb{R}}

\def\hbar{\bar{h}}

\def\dist{{\rm dist}}

\def\leq{\leqslant}
\def\geq{\geqslant}

\newcommand{\rst}[1]{\ensuremath{{\mathbin |}%
\raise-.5ex\hbox{$#1$}}}

\title[Competition-diffusion systems and optimal partition problems]{Sign-changing solutions of competition-diffusion elliptic systems and optimal partition problems}
\author[H. Tavares]{Hugo Tavares}
\address[H. Tavares]{University of Lisbon, CMAF, Faculty of Science, Av. Prof. Gama Pinto 2, 1649-003
Lisboa, Portugal}
\email{htavares@ptmat.fc.ul.pt}

\author[S. Terracini]{Susanna Terracini}
\address[S. Terracini]{Dipartimento di Matematica e Applicazioni, Universit`a degli Studi di Milano-Bicocca,
via Bicocca degli Arcimboldi 8, 20126 Milano, Italy }
\email{susanna.terracini@unimib.it }


\date{\today}

\begin{document}

\selectlanguage{english}

\begin{abstract}
In this paper we prove the existence of infinitely many sign-changing solutions for the system of $m$ Schr\"odinger equations with competition interactions
$$
-\Delta u_i+a_i u_i^3+\beta u_i \sum_{j\neq i} u_j^2 =\lambda_{i,\beta} u_i \quad u_i\in H^1_0(\Omega), \quad i=1,\ldots,m
$$
where $\Omega$ is a bounded domain, $\beta>0$ and $a_i\geq 0\ \forall i.$ Moreover, for $a_i=0$, we show a relation between critical energies associated with this system and the optimal partition problem
$$
\mathop{\inf_{\omega_i\subset \Omega \text{ open}}}_{\omega_i\cap \omega_j=\emptyset\forall i\neq j} \sum_{i=1}^{m} \lambda_{k_i}(\omega_i),
$$
where $\lambda_{k_i}(\omega)$ denotes the $k_i$--th eigenvalue of $-\Delta$ in $H^1_0(\omega)$. In the case $k_i\leq 2$ we show that the optimal partition problem appears as a limiting critical value, as the competition parameter $\beta$ diverges to $+\infty$.
\end{abstract}

\maketitle
\section{Introduction}
Let $\Omega$ be a bounded regular domain in $\R^N$, $N\geq 2$, and let $m\in \N$. In this paper we are concerned with the study of the following system of Schr\"odinger equations with competitive interactions
\begin{equation}\label{eq:Bose-Einstein}
\left\{
\begin{array}{l}
-\Delta u_i+a_i u_i^3+\beta u_i \sum_{j\neq i} u_j^2 =\lambda_{i,\beta} u_i\\[10pt]
u_i\in H^1_0(\Omega), \quad i=1,\ldots,m,
\end{array}
\right.
\end{equation}
where $\beta>0$, $a_i\geq 0$ and $\lambda_{i,\beta}$ are real parameters.

The first purpose of this paper is to prove the following result.
\begin{thm}\label{thm:main1}
For each $\beta>0$ and $a_1,\ldots,a_m\geq 0$, there exist infinitely many sign-changing solutions of \eqref{eq:Bose-Einstein}.
\end{thm}

In this context, a vector solution $u=(u_1,\ldots,u_m)\in H^1_0(\Omega;\R^m)$ is said to be \emph{sign-changing} if $u_i^+,u_i^-\not\equiv 0$ for every $i$. We stress that, for each $\beta>0$, $\lambda_{i,\beta}$ is not fixed \emph{a priori}; instead, by the statement ``$u$ is a solution of \eqref{eq:Bose-Einstein}'' we mean that there exists $(u,\lambda)$ such that \eqref{eq:Bose-Einstein} holds. Assuming enough regularity on the solution, clearly $\lambda_{i,\beta}$ will depend on $u$ through the relation
$$
\lambda_{i,\beta}=\frac{\int_\Omega (|\nabla u_i|^2+a_iu_i^4 + \beta u_i^2 \sum_{j\neq i} u_j^2)\, dx}{\int_\Omega u_i^2\, dx}.
$$

System \eqref{eq:Bose-Einstein} arises in the study of many physical phenomena, such as the study of standing waves in a mixture of Bose-Einstein condensates in $m$ different hyperfine states. The parameters $a_i$ (called the intraspecies scattering length) represent the self-interactions of each state; when $a_i>0$ this is called the defocusing case, in opposition to the focusing one, when $a_i<0$. As for the parameter $\beta$ (the interspecies scattering length), it represents the interaction between unlike particles. Since we assume $\beta>0$, the interaction is of repulsive type.

In the last few years, several mathematical questions have been studied around system \eqref{eq:Bose-Einstein}. When it comes to existence results in bounded domains, all results presented in the literature concern the case of $m=2$ equations and $N\leq 3$. The authors, in collaboration with Noris and Verzini \cite{l2genus}, have shown the existence of positive solutions in the defocusing case $a_1=a_2=a>0$. In the focusing case $a_1=a_2=a<0$, for $\lambda_{1,\beta}\equiv \lambda_{2,\beta}\equiv \lambda<0$ (fixed {\it a priori}), Dancer, Wei and Weth \cite{dancerweiweth} have shown the existence of infinitely many positive solutions of \eqref{eq:Bose-Einstein}, while for $\lambda>0$ the same result was proved by Noris and Ramos \cite{norisramos}. In all these works the fact that the system is invariant under the transformation $(u_1,u_2)\mapsto (u_2,u_1)$ plays a crucial role. We would also like to mention, always in the focusing case, the works by Bartsch, Dancer and Wang \cite{bartschdancerwang} for local and global bifurcation results in terms of the parameter $\beta$, and the results of Domingos and Ramos \cite{domingosramos} concerning the existence of positive solutions for some $\lambda_{1,\beta}\equiv \lambda_1\neq \lambda_2\equiv \lambda_{2,\beta}$. Our existence result of sign-changing solutions for systems of type \eqref{eq:Bose-Einstein} is, up to our knowledge, new.

Another interesting feature of system \eqref{eq:Bose-Einstein} is the asymptotic study of its solutions as $\beta\to +\infty$. Although everything of what we are about to say holds true in a more general framework, let us focus our attention at this point to the case where $a_i=0\ \forall i$ in \eqref{eq:Bose-Einstein}, that is:
\begin{equation}\label{eq:Bose-Einstein_omega_i=0}
\left\{
\begin{array}{l}
-\Delta u_i+\beta u_i \sum_{j\neq i} u_j^2 =\lambda_{i,\beta} u_i\\[10pt]
u_i\in H^1_0(\Omega), \quad i=1,\ldots,m.
\end{array}
\right.
\end{equation}
As mentioned before, $\beta>0$ is of repulsive type, and it has been shown (see for example \cite{CL_regularity, nttv1, weiweth}, among others) that in several situations it occurs what is called phase separation, which means that the limiting profiles (as $\beta\to +\infty$) have disjoint supports. In particular in \cite{nttv1} it is proved that if $\{u_\beta\}_\beta=\{(u_{1,\beta},\ldots, u_{m,\beta})\}_\beta$ is a family of solutions of \eqref{eq:Bose-Einstein_omega_i=0} uniformly bounded in $L^\infty$--norm, and $\{\lambda_{i,\beta}\}_\beta$ is bounded in $\R$ for all $i$, then there exists $\bar u=(\bar u_1,\ldots,\bar u_m)$ such that $\bar u_i\cdot \bar u_j\equiv 0$ in $\Omega$ $\forall i\neq j$ and, up to a subsequence, $u_{i,\beta}\to \bar u_i$ strongly in $H^1_0(\Omega)\cap C^{0,\alpha}(\overline \Omega)$. Moreover,

\begin{equation}\label{eq:limiting_system}
-\Delta \bar u_i=\lambda_i\bar u_i \quad \text{ in the open set } \{u_i\neq 0\},
\end{equation}
with $\lambda_i=\lim_\beta \lambda_{i,\beta}$ (in some sense, \eqref{eq:Bose-Einstein_omega_i=0} can be seen as a singular perturbation of \eqref{eq:limiting_system}). Observe that then $\lambda_i$ is an eigenvalue of $-\Delta$ in $H^1_0(\{\bar u_i\neq 0\})$ and that the sets $\{\bar u_i\neq 0\}$ are disjoint. Therefore it is natural to look for relations between solutions of \eqref{eq:Bose-Einstein_omega_i=0} and solutions of the class of optimal partition problems

\begin{equation}\label{eq:class_of_optimalpartition}
\text{for }k_1,\ldots, k_m\in \N,\qquad \mathop{\inf_{\omega_i\subset \Omega \text{ open}}}_{\omega_i\cap \omega_j=\emptyset\forall i\neq j} \sum_{i=1}^{m} \lambda_{k_i}(\omega_i),
\end{equation}
where $\lambda_{k_i}(\omega)$ denotes the $k_i$--th eigenvalue (counting multiplicities) of $(-\Delta, H^1_0(\omega))$.
The second main result of this paper is the following.
\begin{thm}\label{thm:main2}
Consider \eqref{eq:class_of_optimalpartition} with $k_1=\ldots=k_m=2$, that is
\begin{equation}\label{eq:optimalpartition}
\mathop{\inf_{\omega_i\subset \Omega \text{ open}}}_{\omega_i\cap \omega_j=\emptyset\forall i\neq j} \sum_{i=1}^{m} \lambda_2(\omega_i).
\end{equation}
Then there exist a sequence $u_\beta=(u_{1,\beta},\ldots, u_{m,\beta})$ and a Lipschitz vector function $u=(u_1,\ldots, u_m)$ such that
\begin{itemize}
\item[$(i)$] $u_\beta$ is a sign-changing solution of \eqref{eq:Bose-Einstein_omega_i=0};
\item[$(ii)$] $u_{i,\beta}\to u_i$ in $C^{0,\alpha}(\overline \Omega)\cap H^1_0(\Omega)$ for every $i=1,\ldots, m$, as $\beta\to +\infty$;
\item[$(iii)$] if $\omega_i:=\{u_i\neq 0\}$, then $(\omega_1,\ldots,\omega_m)$ solves \eqref{eq:optimalpartition}.
\end{itemize}
Moreover, we have $\overline \Omega =\cup_{i=1}^m \overline \omega_i$ and the set $\Gamma:=\Omega\cap (\cup_{i=1}^m \partial \omega_i)$ is a regular hypersurface of class $C^{1,\alpha}$, up to a set having at most Hausdorff measure $N-2$.
\end{thm}

Adapting the proof of the previous theorem, we will actually see that a similar result holds for \eqref{eq:class_of_optimalpartition} in the more general case where one takes $k_1,\ldots, k_m\in \{1,2\}$ (with the difference that the approximating solutions $u_\beta$ will only change sign in the components $i$ such that $k_i=2$, and all the other components will be positive). We should mention that for $k_1=\ldots=k_m=1$ a result similar to Theorem \ref{thm:main2} was already know if one combined the papers \cite{CL_regularity, CL_eigenvalues}. Some preliminary results were also proved by Conti, Terracini and Verzini \cite{ctv4}, while Helffer, Hoffmann-Ostenhof and Terracini \cite{HHOT} have proved that, in dimension two, \emph{every} solution $(\omega_1,\ldots,\omega_m)$ of \eqref{eq:class_of_optimalpartition} is regular in the sense of the last paragraph of Theorem \ref{thm:main2}. Passing from the case of a sum of first eigenvalues to the sum of second eigenvalues is not trivial, because while in the first case one can work with minima of the the energy functional associated with \eqref{eq:Bose-Einstein_omega_i=0}, in the latter case one has to define an appropriate minimax quantity. We would like to mention that in the case $k_1=\ldots=k_m=k$, the existence of solution of \eqref{eq:class_of_optimalpartition} was proved in the class of quasi-open sets by Bucur, Buttazzo and Henrot \cite{BBH}, and more recently in the class of open sets by Burdin, Bucur and Oudet \cite{BBO}.

The structure of this paper is as follows. In Section \ref{sec:second_section} we prove the existence of infinitely many sign-changing solutions for a general competitive system. The main tool will be the use of a new notion of Krasnoselskii genus, which will take in consideration the fact that the functionals considered are even in each single component. This genus will be rather effective in connecting problem \eqref{eq:Bose-Einstein_omega_i=0} with \eqref{eq:class_of_optimalpartition} (as will become evident in Lemma \ref{lemma:relaxed_c_infty}). Section \ref{sec:third_section} is then dedicated to the proof of Theorem \ref{thm:main1}, applying the results of Section \ref{sec:second_section} to system \eqref{eq:Bose-Einstein}. Observe that one difficulty to overcome is the fact that the energy functional
$$
u=(u_1,\ldots,u_m) \mapsto \sum_{i=1}^m \int_\Omega  (|\nabla u_i|^2 + \frac{a_i u_i^4}{2})\, dx + \mathop{\sum_{i,j=1}^m}_{i\neq j} \frac{\beta}{2}\int_\Omega  u_i^2 u_j^2 \, dx
$$
for $\|u_i\|_{L^2(\Omega)}=1$, might take the value $+\infty$. We overcome this fact by using a truncation argument. Finally in the last section we will present the proof of Theorem \ref{thm:main2}.

\section{Sign-changing solutions for general competitive systems}\label{sec:second_section}

Take two odd functions $f,g:\R\to \R$, of class $C^1$, such that
\begin{itemize}
\item[($fg1$)] $f'(t),g'(t)\geq 0$ for every $t>0$;
\item[($fg2$)] There exist $C>0$ and $1<p< \min\{2^\ast/2,3\}$, $1<q<\min\{2^*,3\}$ such that
$$
|f(t)|\leq C(1+|t|^{p-1}),\quad |g(t)|\leq C(1+|t|^{q-1}) \qquad \text{ for every $t \in \R$. }
$$
\item[($fg3$)] For every $s,t\geq 0$,
$$
f(s)t+f(t)s\leq f(s)s+f(t)t \qquad \text{and}\qquad g(s)t+g(t)s\leq g(s)s+g(t)t.
$$
\end{itemize}
Let  $G(s):=\int_0^s g(\xi)\, d\xi$, $F(s):=\int_0^s f(\xi)\, d\xi.$

In this section we will focus on the proof of the following result.

\begin{thm}\label{thm:infinitely_many_general}
There exist infinitely many sign-solutions of the system
\begin{equation}\label{eq: BEC_more_general}
-\Delta u_i+g(u_i)+f(u_i)\sum_{j\neq i} F(u_j)=\lambda_i u_i, \ \ u_i\in H^1_0(\Omega), \quad i=1,\ldots,m.
\end{equation}
\end{thm}

As we shall see in the next section, Theorem \ref{thm:main1} will be a consequence of this theorem.
\begin{rem}\label{rem:genus_usual_definition} From the previous list of hypotheses, we can conclude that $F(t), G(t)$ are even nonnegative functions, $f(t)t, f'(t), g(t)t, g'(t)\geq 0$ for every $t\in \R$, $f(0)=g(0)=0$, and moreover
\begin{equation}\label{eq:f(s)t+f(t)s}
f(s)t+f(t)s\leq f(s)s+f(t)t \text{ and } g(s)t+g(t)s\leq g(s)s+g(t)t\quad \text{ for every } s,t\in \R.
\end{equation}
\end{rem}

We will look for solutions of \eqref{eq: BEC_more_general} as critical points of the functional
$$
J(u)=\sum_{i=1}^m \int_\Omega ( |\nabla u_i|^2+2G(u_i))\, dx+\mathop{\sum_{i,j=1}^m}_{j\neq i} \int_\Omega F(u_i)F(u_j)\, dx.
$$
restricted to the $L^2$--sphere
$$
\Mah=\{u=(u_1,\ldots,u_m)\in H^1_0(\Omega;\R^m):\ \|u_i\|_{L^2(\Omega)}=1 \ \forall i\}.
$$
In order to obtain infinitely many critical points, we will define several minimax levels using a new definition of \emph{vector genus}.


\subsection{Vector genus. Minimax levels}

Take the involutions
$$
   \sigma_{i}:\Mah\to \Mah,\qquad \sigma_i(u_1,\ldots, u_m)= (u_1,\ldots, -u_i,\ldots, u_m)\quad \forall i.
$$
Consider moreover the class of sets
$$
\Feh=\{A\subseteq \Mah: \ A \text{ is a closed set and } \sigma_i(u)\in A\ \forall u\in A,\  i=1,\ldots,m  \}
$$
and, for each $A\in \Feh$ and $k_1,\ldots, k_m\in \N$, the class of functions
$$
F_{(k_1,\ldots,k_m)}(A)=\left\{ f=(f_1,\ldots,f_m):A\to \prod_{i=1}^m \R^{k_i-1}:
\begin{array}{l}
f_i:A\to \R^{k_i-1} \text{ continuous, and }\\
f_i(\sigma_i(u))=-f_i(u) \text{ for every }i\\
f_i(\sigma_j(u))=f_i(u) \text{ whenever }j\neq i
\end{array}
\right\}.
$$
\begin{defin}[vector genus]
Let $A\in \Feh$ and take $m$ positive integers $k_1,\ldots, k_m$. We say that $\vec{\gamma}(A)\geq (k_1,\ldots,k_m)$ if for every $f\in F_{(k_1,\ldots,k_m)}(A)$ there exists $\bar u\in A$ such that $f(\bar u)=(f_1(\bar u),\ldots, f_m(\bar u))=(0,\ldots,0).$ We denote
$$
\Gamma^{(k_1,\ldots,k_m)}:=\{A\in \Feh:\ \vec{\gamma}(A)\geq (k_1,\ldots,k_m)\}.
$$
\end{defin}
\begin{rem}
Observe that we don't actually define the quantity $\vec{\gamma}(C)$, but only give a meaning to the expression ``$\vec{\gamma}(C)\geq (k_1,\ldots,k_m)$''.
\end{rem}

\begin{rem}
Recall the usual definition of Krasnoselskii genus associated with the $\Z_2$ symmetry group: for every nonempty and closed set $A\subset H^1_0(\Omega)$ such that $-A=A$,
$$
\gamma(A):=\inf\{k:\ \text{ there exists } h:A\to \R^k\setminus\{0\} \text{ continuous and odd}\}
$$
and $\gamma(A):=\infty$ if no such $k$ exists. Then for $m=1$ the notion of vector genus coincides with the usual one, in the sense that, for $k\in \N$,
$$
\vec \gamma(A)\geq k \iff \gamma(A)\geq k.
$$
\end{rem}

The key properties of this notion of genus will come out from the following Borsuk-Ulam type result due to Dzedzej, Idzik and Izydorek (see \cite{Borsuk_Ulam_I,Borsuk_Ulam_II}). A weaker version for the case of the product of two spheres had already been proved by Zhong \cite{Borsuk_Ulam_twospheres}.

\begin{thm}\label{thm:borsuk_ulam}
If $\tilde f: \prod_{i=1}^m S^{n_i}\to \prod_{i=1}^m \R^{n_i}$ is a continuous function such that, for every $i\in \{1,\ldots,m$\},
\begin{eqnarray*}
\tilde f_i(x_1,\ldots, -x_i,\ldots,x_m)=-\tilde f_i(x_1,\ldots,x_i,\ldots, x_m),\\
\tilde f_i(x_1,\ldots, -x_j,\ldots,x_m)=\tilde f_i(x_1,\ldots,x_j,\ldots, x_m)\ \forall j\neq i,
\end{eqnarray*}
then there exists $(\bar x_1,\ldots, \bar x_m)\in \prod_{i=1}^m S^{n_i}$ such that $\tilde f(\bar x_1,\ldots, \bar x_m)=(0,\ldots,0).$
\end{thm}

\begin{lemma}\label{lem:properties_of_vector_genus}
With the previous notations, the following properties hold.
\begin{enumerate}
\item[$(i)$] Take $\prod_{i=1}^m A_i\subseteq \Mah$ and let $\eta_i:S^{k_i-1}\subset \R^{k_i} \to A_i$ be a homeomorphism such that $\eta_i(-x)=-\eta_i(x)$ for every $x\in S^{k_i-1}$, $i\in \{1,\ldots,m\}$.
Then $$\prod_{i=1}^m A_i \in \Gamma^{(k_1,\ldots,k_m)}.$$

\item[$(ii)$] We have $\overline {\eta(A)} \in \Gamma^{(k_1,\ldots,k_m)}$ whenever $A\in \Gamma^{(k_1,\ldots,k_m)}$ and $\eta:A\to \Mah$ is such that $\eta\circ \sigma_i=\sigma_i\circ \eta\ \forall i$.
\end{enumerate}
\end{lemma}
\begin{proof}i) Take $f\in F_{(k_1,\ldots,k_m)}(\prod_{i=1}^m A_i)$ and consider the map
$$
\vphi:\prod_{i=1}^m S^{k_i-1} \to \prod_{i=1}^m\R^{k_i-1}; \quad \vphi(x_1,\ldots, x_m):=f(\eta_1(x_1),\ldots, \eta_m(x_m)).
$$
For each fixed $i\in \{1,\ldots, m\}$, we have that
\begin{eqnarray*}
\vphi_i(x_1,\ldots,-x_i,\ldots, x_m)&=&  f_i(\eta_1(x_1),\ldots, \eta_i(-x_i),\ldots, \eta_m(x_m))\\
						    &=&  f_i(\eta_1(x_1),\ldots, -\eta_i(x_i),\ldots, \eta_m(x_m))\\
						    &=& -f_i(\eta_1(x_1),\ldots, \eta_i(x_i),\ldots, \eta_m(x_m))\\
						    &=& -\vphi_i(x_1,\ldots, x_i,\ldots, x_m)
\end{eqnarray*}
and, for $j\neq i$,
\begin{eqnarray*}
\vphi_i(x_1,\ldots,-x_j,\ldots, x_m)&=&  f_i(\eta_1(x_1),\ldots, \eta_j(-x_j),\ldots, \eta_m(x_m))\\
						    &=&  f_i(\eta_1(x_1),\ldots, -\eta_j(x_j),\ldots, \eta_m(x_m))\\
						    &=& f_i(\eta_1(x_1),\ldots, \eta_j(x_j),\ldots, \eta_m(x_m))\\
						    &=& \vphi_i(x_1,\ldots, x_j,\ldots, x_m).
\end{eqnarray*}
Hence Theorem \ref{thm:borsuk_ulam} implies that $\vphi^{-1}(\{(0,\ldots,0)\})\neq \emptyset$, and hence also $f^{-1}(\{(0,\ldots,0)\})\neq \emptyset$.

ii) First of all, it is easy to prove that if $A\in \Feh$ and $\eta$ is as in the statement, then the set $\overline{\eta(A)}\in \Feh$. Take any $f\in F_{(k_1,\ldots,k_m)}(\overline{\eta(A)}).$ Then the map
$$
f\circ\eta:A\to \prod_{i=1}^m \R^{k_i-1},\quad u\mapsto (f_1(\eta(u)),\ldots,f_m(\eta(u)))
$$
is continuous and, for every $i$,
\begin{equation*}
f_i(\eta(\sigma_i(u))= f_i(\sigma_i(\eta(u)))=-f_i(\eta(u)),
\end{equation*}
and for every $i\neq j$
$$
f_i(\eta(\sigma_j(u)))=f_i(\sigma_j(\eta(u)))=f_i(\eta(u)).
$$

Hence $f\circ \eta \in F_{(k_1,\ldots,k_m)}(A)$ and from the definition of genus we deduce the existence of $\bar u\in A$ such that $f(\eta(\bar u))=(0,\ldots,0)$, and the proof is complete.
\end{proof}

Together with this notion of genus, in order to obtain solutions which change sign, we will use a strategy based on the work of Conti, Merizzi, Terracini \cite{contimerizziterracini}, using cones of positive/negative functions. A similar approach was also used for instance in \cite{wethetal, ramostavzou}.
In our case, for each $i\in \{1,\ldots,m\}$, we define the cone
$$\Peh_i=\{u=(u_1,\ldots,u_m)\in H^1_0(\Omega; \R^m):\ u_i\geq 0\}$$
and take $\Peh:=\cup_{i=1}^m (\Peh_i\cup -\Peh_i)$. Moreover, for each $\delta>0$, we define $\Peh_\delta=\{u\in H^1_0(\Omega; \R^m): \dist_2(u,\Peh)<\delta\}$, where $\dist_2$ denotes the distance associated with the $L^2$--norm. Observe that $\dist_2(u,\Peh_i)=\|u_i^-\|_{L^2(\Omega)}$ and $\dist_2(u,-\Peh_i)=\|u_i^+\|_{L^2(\Omega)}$.

\begin{lemma}\label{lem:A_setminus_Cone_is_nonempty}
For every $\delta<\sqrt{2}/2$ we have that $A\setminus \Peh_\delta \neq \emptyset$ whenever $A\in \Gamma^{(k_1,\ldots,k_m)}$ with $k_i\geq 2\ \forall i$.
\end{lemma}
\begin{proof}
Given $A\in \Gamma^{(k_1,\ldots,k_m)}$, consider the map
$$
f=(f_1,\ldots,f_m):A\to \prod_{i=1}^m\R^{k_i-1};\qquad f_i(u)= \Big( \int_\Omega u_i|u_i|\, dx,0,\ldots,0\Bigr).
$$
Clearly $f\in F_{(k_1,\ldots,k_m)}(A)$, hence there exists $\bar u\in A$ such that $f(\bar u)=(0,\ldots,0)$. By recalling that $A\subseteq \Meh$, we deduce that
$$\int_\Omega (\bar u_i^+)^2\, dx =\int_\Omega (\bar u_i^-)^2 \, dx=\frac{1}{2} \qquad \text{ for all $i$}.$$
Thus $\dist_2(\bar u,\Peh)=\sqrt{2}/2$, and $\bar u\in A\setminus \Peh_\delta$ for every $\delta<\sqrt{2}/2$.
\end{proof}

We are now ready to define a sequence of minimax levels which will turn out to be critical levels for $J|_\Mah$. For every $k_1,\ldots, k_m\geq 2$ and $\delta<\sqrt{2}/2$, define
\begin{equation}\label{eq:general_minmax_level}
d_\delta^{k_1,\ldots,k_m}=\inf_{A\in \Gamma^{(k_1,\ldots,k_m)}}\sup_{A\setminus \Peh_\delta} J.
\end{equation}

\begin{rem}
It will be important to have an upper-bound for these minimax levels which is independent of $\delta$. Considering
\begin{equation*}
{\tilde d}^{k_1,\ldots,k_m}=\inf_{A\in \Gamma^{(k_1,\ldots,k_m)}}\sup_{A} J_\beta,
\end{equation*}
it is easy to see that
$$
d_{\delta}^{k_1,\ldots,k_m}\leq {\tilde d}^{k_1,\ldots,k_m} \qquad \text{ for every $k_1,\ldots, k_m\in \N$, $\delta>0$}.
$$
Throughout this chapter, we will denote ${\tilde d}^{k_1,\ldots,k_m}$ simply by $\tilde d$.
\end{rem}

\subsection{Existence of sign-changing critical points of $J|_\Mah$ at level $d_\delta^{k_1,\ldots, k_m}$}
As a first step towards the proof of Theorem \ref{thm:infinitely_many_general}, we will now show that $d_\delta^{k_1,\ldots,k_m}$ is indeed a critical level for sufficiently small $\delta$. More precisely, we have the following.
\begin{thm}\label{thm:sign-changing_solution}
There exists $\delta>0$, $u\in H^1_0(\Omega;\R^m)$ and $\lambda_i\in \R$ such that
$$
-\Delta u_i +g(u_i)+f(u_i)\sum_{j\neq i} F(u_j)=\lambda_i u_i \quad\text{ in } \Omega,\qquad i=1,\ldots,m
$$
and $J(u)=d_\delta^{k_1,\ldots, k_m}$. Moreover, each $u_i$ is a sign-changing function.
\end{thm}

In order to prove this result we need to find a pseudogradient for $J$ over $\Mah$ for which the set $\Peh_\delta$ is positively invariant for the associated flow. Following \cite[Theorem 3.1]{contimerizziterracini}, such pseudogradient should be of the type $Id-K$, where $Id$ is the identity in $H^1_0(\Omega;\R^m)$ and $K$ is an operator such that $K(\Peh_\delta)\subseteq \Peh_{\delta/2}$ for small $\delta$. The gradient of $J$ constrained to $\Mah$ does not seem to satisfy this, due to the sign of the terms $G(u_i),\sum_{j\neq i}F(u_i)F(u_j)$, and hence this part is not straightforward.

For technical reasons,  we will work on the neighborhood of $\Mah$ in $H^1_0(\Omega;\R^m)$:
 $$
 \Mah^\ast=\{u\in H^1_0(\Omega;\R^m): \ \|u_i\|_{L^2(\Omega)}>\frac{1}{2} \forall i \}
 $$
 (observe that $u_i\not\equiv 0\ \forall i$ whenever $u\in \Mah^\ast$).

\begin{prop}
Given $u\in \Mah^\ast$ and $i\in \{1,\ldots,m\}$, there exists a unique solution $w_i\in H^1_0(\Omega)$, $\mu_i\in \R$ of the problem
\begin{equation}\label{eq:existence_and_uniqueness_for_K}
\left\{
\begin{array}{l}
-\Delta w_i +g(w_i)+ f(w_i)\sum_{j\neq i} F(u_j)=\mu_i u_i \ \text{ in } \Omega,\\[8pt]
\displaystyle \int_\Omega u_iw_i\, dx =1.
\end{array}
\right.
\end{equation}
\end{prop}
\begin{proof}
\emph{Existence:}
Fix $u\in \Mah^\ast$ and consider the minimization problem
$$
m:=\inf \Bigl\{ \int_\Omega (\frac{1}{2}|\nabla w|^2+ G(w)+F(w)\sum_{j\neq i}F(u_j))\, dx: w\in H^1_0(\Omega),\ \int_\Omega w u_i\, dx=1 \Bigr\} \geq 0.
$$
Take a minimizing sequence $(w_n)_n$, that is
$$\int_\Omega (\frac{1}{2}|\nabla w_n|^2+G(w_n)+F(w_n)\sum_{j\neq i}F(u_j))\, dx \to m \quad\text{ and }\quad \int_\Omega w_n u_i\, dx=1.$$
As $F,G$ are nonnegative function, we obtain that $(w_n)_n$ is a $H^1_0$--bounded sequence, thus there exists $\bar w$ such that, up to a subsequence,
$$
w_n\rightharpoonup \bar w \qquad \text{ weakly in } H^1_0(\Omega), \text{ and strongly in } L^2(\Omega) \text{ and } L^{2p}(\Omega).
$$
Therefore $\int_\Omega \bar w u_i\, dx =1$ and
\begin{eqnarray*}
m &\leq& \int_\Omega (\frac{1}{2}|\nabla \bar w|^2+G(\bar w)+F(\bar w)\sum_{j\neq i}F(u_j))\, dx\\
     &\leq& \liminf_{n\to \infty}  \int_\Omega (\frac{1}{2}|\nabla w_n|^2+G(w_n)+F(w_n)\sum_{j\neq i}F(u_j))\, dx=m.
\end{eqnarray*}
Thus $\bar w$ achieves $m$, and by the Lagrange multiplier rule we have that $\bar w$ solves \eqref{eq:existence_and_uniqueness_for_K} for some $\mu_i$.\\

\emph{Uniqueness:} Take $w$ and $v$ to be solutions of
$$
-\Delta w + g(w)+f(w)\sum_{j\neq i} F(u_j)=\mu_1 u_i,\qquad \int_\Omega w u_i\, dx =1
$$
and
$$
-\Delta v + g(v)+f(v)\sum_{j\neq i} F(u_j)=\mu_2 u_i,\qquad \int_\Omega v u_i\, dx =1.
$$
Subtracting the second equation from the first one, multiplying the result by $w-v$ and integrating by parts yields
\begin{multline*}
\int_\Omega |\nabla (w-v)|^2\, dx +\int_\Omega (g(w)-g(v))(w-v)\, dx+ \int_\Omega (f(w)-f(v))(w-v) \sum_{j\neq i} F(u_j) \, dx\\
=\int_\Omega \mu_1 u_i(v-w)\, dx - \int_\Omega \mu_2 u_i (v-w)\, dx =0.
\end{multline*}
As $s\mapsto f(s),g(s)$ are non-decreasing (\emph{cf.} ($fg1$)), then $(f(w)-f(v))(w-v)\geq 0$ and $(g(w)-g(v))(w-v)\geq 0$, whence
$$
\int_\Omega |\nabla (w-v)|^2\, dx=0, \qquad \text{ and } \qquad w\equiv v.
$$
Finally, observe that from \eqref{eq:existence_and_uniqueness_for_K} we deduce that each $\mu_i$ is uniquely determined by the expression
\begin{equation}\label{eq:mu_i_function_of_u_w}
\mu_i=\int_\Omega (|\nabla w_i|^2+g(w_i)w_i+f(w_i)w_i \sum_{j\neq i}F(u_j))\, dx.
\end{equation}
\end{proof}

We can now define the operator
$$
K:\Mah^\ast\to H^1_0(\Omega;\R^m);\qquad u \mapsto K(u)=w,
$$
that is, for each $u$, $K(u)=w$ is the unique solution of the system \eqref{eq:existence_and_uniqueness_for_K}.

Next we state and prove three properties of the operator $K$.

\begin{lemma}\label{lem:K_is_compact} ($K|_\Mah$ is a compact operator) Let $(u_n)_n\subset \Mah$ be a bounded sequence in $H^1_0(\Omega;\R^m)$. Then there exists $w\in H^1_0(\Omega;\R^m)$ such that, up to a subsequence,
$$
K(u_n)\to w \qquad \text{ strongly in } H^1_0(\Omega;\R^m).
$$
\end{lemma}

\begin{proof}
Let $(u_n)_n=(u_{1,n},\ldots,u_{m,n})_n$ be as in the statement and let $w_n:=K(u_n)$. Multiplying \eqref{eq:existence_and_uniqueness_for_K} by $u_{i,n}$ and integrating by parts, we deduce

\begin{eqnarray*}
&&\int_\Omega |\nabla w_{i,n}|^2 \leq \int_\Omega (|\nabla w_{i,n}|^2+g(w_{i,n})w_{i,n}+f(w_{i,n})w_{i,n}\sum_{j\neq i}F(u_{j,n}))\, dx=\mu_{i,n}\\
&& \quad = \int_\Omega (\nabla w_{i,n}\cdot \nabla u_{i,n} +g(w_{i,n})u_{i,n}+f(w_{i,n})u_{i,n}\sum_{j\neq i}F(u_{j,n}) )\, dx\\
&& \quad						 \leq  \|w_{i,n}\|_{H^1_0(\Omega)}\| u_{i,n}\|_{H^1_0(\Omega)}+  C_1 \int_\Omega (1+|w_{i,n}|^{q-1})|u_{i,n}|\, dx\\
&&\qquad+C_1\int_\Omega (1+|w_{i,n}|^{p-1})|u_{i,n}|(1+\sum_{j\neq i}|u_{j,n}|^p  )\, dx\\
&& \quad \leq C_2\|w_{i,n}\|_{H^1_0(\Omega)} + C_1\int_\Omega (|u_{i,n}| + |w_{i,n}|^{q-1}|u_{i,n}|)\,dx+ \\
&& \qquad C_1\int_\Omega (|u_{i,n}|+|w_{i,n}|^{p-1} |u_{i,n}| +|u_{i,n}|\sum_{j\neq i}|u_{j,n}|^p +  |w_{i,n}|^{p-1}|u_{i,n}|\sum_{j\neq i}|u_{j,n}|^p )\, dx \\
&& \quad \leq C_2\|w_{i,n}\|_{H^1_0(\Omega)} + C_3+ C_1 \|u_{1,n}\|_{L^q(\Omega)}\|w_{i,n}\|_{L^q(\Omega)}^{q-1} +C_1 \|u_{i,n}\|_{L^p(\Omega)} \|w_{i,n}\|^{p-1}_{L^p(\Omega)}\\
&& \qquad + C_1\|u_{i,n}\|_{L^2(\Omega)}\sum_{j\neq i}\|u_{j,n}\|_{L^{2p}(\Omega)}^p +C_1 \|w_{i,n}\|^{p-1}_{L^{2p}(\Omega)}\|u_{i,n}\|_{L^{2p}(\Omega)}\sum_{j\neq i} \| u_{j,n}\|_{L^{2p}(\Omega)}^p\\
&&\quad 						 \leq C_2\|w_{i,n}\|_{H^1_0(\Omega)} + C_3 + C_4  \|w_{i,n}\|^{q-1}_{H^1_0(\Omega)}+C_5  \|w_{i,n}\|^{p-1}_{H^1_0(\Omega)}.
\end{eqnarray*}
As $p,q<3$, then $p-1,q-1<2$ and we conclude that $(w_n)_n$ is $H^1_0$--bounded. Hence also all $\mu_{i,n}$ are bounded (recall from \eqref{eq:mu_i_function_of_u_w} their expressions) and, up to a subsequence, $w_n$ converges weakly in $H^1_0$ to some function $\bar w$.
Multiplying this time row $i$ in \eqref{eq:existence_and_uniqueness_for_K} by $w_{i,n}-\bar w_i$, we see that
\begin{multline*}
\int_\Omega \nabla w_{i,n}\cdot \nabla (w_{i,n}-\bar  w_i)\, dx=-\int_\Omega g(w_{i,n})( w_{i,n}-\bar w_i)\, dx -\int_\Omega f(w_{i,n}) ( w_{i,n}-\bar w_i) \sum_{j\neq i}F( u_{j,n}) \, dx \\
													+ \int_\Omega \mu_{i,n}u_{i,n} ( w_{i,n}-\bar w_i)\, dx\to 0,
\end{multline*}
and therefore $w_{i,n}\to \bar w_i$ strongly in $H^1_0(\Omega)$.
\end{proof}

\begin{lemma}
The operator $K$ is of class $C^1$.
\end{lemma}
\begin{proof}
We will apply the Implicit Function Theorem to the $C^1$ map
\begin{eqnarray*}
&\Psi:\Mah^\ast \times H^1_0(\Omega)\times \R\to H^1_0(\Omega)\times \R; \\
&\Psi(u,v,\lambda)=\Bigl(v+(-\Delta)^{-1}(g(v)+f(v)\sum_{j\neq i}F(u_j)-\lambda u_i),\int_\Omega v u_i \, dx-1\Bigr).
\end{eqnarray*}
Observe that \eqref{eq:existence_and_uniqueness_for_K} holds if and only if $\Psi(u,w_i,\mu_i)=(0,0)$. Take such a zero of $\Psi$ and let us compute the derivative of $\Psi$ with respect to $v,\lambda$ at the point $(u,w_i,\mu_i)$ in the direction $(\bar w,\bar \lambda)$. We obtain a map $\Phi:H^1_0(\Omega)\times \R\to H^1_0(\Omega)\times \R$ given by
\begin{eqnarray*}
\Phi(\bar w,\bar \lambda)  &:=&D_{v,\lambda} \Psi(u,w_i,\mu_i)(\bar w,\bar \lambda)\\
					 &=&\Bigl(\bar w+(-\Delta)^{-1}(\bar wg'(w_i)+ \bar wf'(w_i)\sum_{j\neq i}F(u_j)-\bar \lambda u_i),\int_\Omega \bar w u_i\, dx \Bigr).
\end{eqnarray*}
Let us prove that $\Phi$ is a bijective map.

\emph{$\Phi$ is injective:} If $\Phi(\bar w,\bar \lambda)=(0,0)$, then we can multiply the equation
\begin{equation}\label{eq:Phi_injective}
-\Delta \bar w+\bar wg'(w_i)+ \bar w  f'(w_i)\sum_{j\neq i}F(u_j)-\bar \lambda u_i=0
\end{equation}
by $\bar w$, yielding
$$
\|\bar w\|_{H^1_0(\Omega)}^2\leq  \|\bar w\|_{H^1_0(\Omega)}^2 +\int_\Omega \bar w^2g'(w_i)\, dx + \int_\Omega \bar w^2 f'(w_i) \sum_{j\neq i}F(u_j)\, dx =\bar \lambda \int_\Omega u_i \bar w\, dx=0,
$$
whence $\bar w\equiv 0$. Again by using \eqref{eq:Phi_injective} we obtain $\bar \lambda u_i=0$, thus also
$$
\bar \lambda=\bar \lambda \int_\Omega u_i w_i\, dx=0.
$$

\emph{$\Phi$ is surjective}: Take $(f,c)\in H^1_0(\Omega)\times \R$ and let $\tilde w_1,\tilde w_2$ be solutions of the (linear) problems
\begin{eqnarray*}
-\Delta \tilde w_1+ \tilde w_1g'(w_i)+\tilde w_1 f'(w_i)\sum_{j\neq i}F(u_j)=f\\
-\Delta \tilde w_2+\tilde w_2g'(w_i)+\tilde w_2 f'(w_i)\sum_{j\neq i}F(u_j)=u_i.
\end{eqnarray*}
Take moreover $\kappa=\Bigl(c-\int_\Omega \tilde w_1 u_i\, dx\Bigr)/ \int_\Omega  \tilde w_2 u_i  \, dx$. Then $\Phi(\tilde w_1+\kappa \tilde w_2,\kappa)=(f,c).$
\end{proof}

\begin{lemma}\label{lem:K_contracts_neigh_of_cone}
There exists $\delta>0$ (which can be chosen arbitrary small) such that
\begin{equation}\label{eq:K_contracts_neigh_of_cone}
\dist_2(K(u),\Peh)<\delta/2,\qquad \forall\ u\in \Mah,\ J(u)\leq \tilde d+1,\ \dist_2(u,\Peh)<\delta.
\end{equation}
\end{lemma}

\begin{proof}
1. Suppose, in view of a contradiction, that there exists $\delta_n\to 0$ and $u_n\in \Mah$ with $J(u_n)\leq \tilde d+1$, $\dist_2(u_n,\Peh)<\delta_n$ and $\dist_2(K(u_n),\Peh)\geq \delta_n/2$. Suppose moreover, without loss of generality, that
$$
\dist_2(u_n,P)=\|u_{1,n}^-\|_2 \ (<\delta_n\to 0).
$$
Let $w_n=K(u_n)$ and $\mu_{i,n}:=\int_\Omega (|\nabla w_{i,n}|^2+g(w_{i,n})w_{i,n}+f(w_{i,n})w_{i,n}\sum_{j\neq i}F(u_{j,n})) \, dx $ for every $i$. We deduce from Lemma \ref{lem:K_is_compact} the existence of $\bar u$, $\bar w$, and $\bar \mu_i$ such that
$$
u_n\to \overline u \quad \text{ weakly in $H^1_0(\Omega;\R^m)$,\ strongly in $L^2(\Omega;\R^m)$ and $L^{2p}(\Omega;\R^m)$};
$$
$$
w_n\to \bar w \quad \text{ strongly in $H^1_0(\Omega;\R^m)$, and $\mu_{i,n}\to \bar \mu_i$ in $\R$.}
$$
Observe that
$$
-\Delta \bar w_1 +g(\bar w_1)+ f(\bar w_1) \sum_{j\geq 2} F(\bar u_j) =\bar  \mu_1 \bar u_1\geq 0,
$$
and (from the hypotheses made on $f,g$) $f(s),g(s)={\rm O}(s)$ as $s\to 0$. Hence by the strong maximum principle $\bar w_1>0$, and therefore we can conclude that $|\{w_{1,n}<0\}|\to 0$ as $n\to \infty$.

2. Observe now that in general, by using both H\"older and Sobolev inequalities we have
$$
\|u\|_{L^2(\Omega)}^2\leq C^2_S\, p(|\Omega|)\|u\|_{H^1_0(\Omega)}^2,
$$
where $p(|\Omega|)=|\Omega|^{(2^\ast-2)/2^\ast}$ and $C_S$ is the best Sobolev constant of the embedding $H^1_0(\Omega)\hookrightarrow L^{2^\ast}(\Omega)$. This fact together with \eqref{eq:existence_and_uniqueness_for_K} allows us to obtain
\begin{eqnarray*}
\| w_{1,n}^-\|_{L^2(\{w_{1,n}<0\})}^2 &\leq & C^2_S\,p(|\{w_{1,n}<0\}|) \int_\Omega |\nabla w_{1,n}^-|^2\, dx\\
						   &\leq& C^2_S\,p(|\{w_{1,n}<0\}|) \int_\Omega (|\nabla w_{1,n}^-|^2-g(w_{1,n})w_{1,n}^--f(w_{1,n})w_{1,n}^- \sum_{j\geq 2}F(u_{j,n}))\, dx\\
						   &=& -\mu_{1,n} C^2_S\,p(|\{w_{1,n}<0\}|) \int_\Omega u_{1,n} w_{1,n}^-\, dx\\
						   &\leq& \mu_{1,n} C^2_S\, p(|\{w_{1,n}<0\}|) \int_\Omega u_{1,n}^- w_{1,n}^-\, dx \\
						   &\leq& \mu_{1,n} C^2_S\, p(|\{w_{1,n}<0\}|) \|u_{1,n}^-\|_{L^2(\Omega)}\|w_{1,n}^-\|_{L^2(\{w_{1,n}<0\})}\\
						   &\leq & C'\, p(|\{w_{1,n}<0\}|) \, \delta_n\, \|w_{1,n}^-\|_{L^2(\{w_{1,n}<0\})}
\end{eqnarray*}
and hence $\|w_{1,n}^-\|_{L^2(\Omega)}<\delta_n/2$ for sufficiently large $n$, which is a contradiction.
\end{proof}

Now define
$$
V:\Mah^\ast \to H^1_0(\Omega; \R^m); \qquad u\mapsto u-K(u).
$$
Observe that, for $u\in \Mah$,
$$
V(u)=0 \iff u \text{ solves } \eqref{eq: BEC_more_general}.
$$
Next we show that $V$ satisfies the Palais-Smale condition and that it is a pseudogradient for $J$ over $\Mah$.

\begin{lemma}[Palais-Smale type condition]\label{lem:palais_smale} Let $u_n\in \Mah$ be such that, as $n\to \infty$,
$$
J(u_n)\to c <\infty \qquad \text{ and } \qquad V(u_n)\to 0 \text{ in } H^1_0(\Omega;\R^m).
$$
Then there exists $u\in \Mah$ such that, up to a subsequence, $u_n\to u$ in  $H^1_0(\Omega;\R^m).$
\end{lemma}

\begin{proof}
Since $\|u_n\|_{H^1_0(\Omega)}^2\leq J(u_n)\leq c+1<\infty$ for large $n$, there exists $u\in \Mah$ and $w\in H^1_0(\Omega; \R^m)$ such that, up to a subsequence,
$$
u_n\rightharpoonup u \text{ weakly in } H^1_0(\Omega;\R^m), \quad \text{ and } w_n:=K(u_n)\to w \text{ strongly in } H^1_0(\Omega;\R^m).
$$
Then we have, as $n\to \infty$,
\begin{equation*}
{\rm o}(1)=\langle V(u_n),u_n-u\rangle_{H^1_0(\Omega)}=\langle u_n,u_n-u \rangle_{H^1_0(\Omega)}-\langle w_n,u_n-u\rangle_{H^1_0(\Omega)}
\end{equation*}
Since the last term tends to zero as $n\to \infty$, the proof is finished.
\end{proof}

\begin{lemma}\label{lem:V_is_pseudogradient}
We have
$$
\langle \nabla J(u),V(u) \rangle_{H^1_0(\Omega)}\geq  2\|V(u)\|_{H^1_0(\Omega)}^2 \qquad \text{ whenever } u\in \Mah.
$$
\end{lemma}
\begin{proof}
First of all observe that, by \eqref{eq:existence_and_uniqueness_for_K},
$$
\int_\Omega u_i(u_i-w_i)\, dx=\int_\Omega u_i^2\, dx-\int_\Omega u_i w_i \, dx=1-1=0 \qquad \text{ whenever }u\in \Mah.
$$
This together with \eqref{eq:f(s)t+f(t)s} yields
\begin{eqnarray*}
\langle \nabla J(u),V(u)\rangle_{H^1_0(\Omega)}&=& 2\sum_{i=1}^m \int_\Omega(\nabla u_i\cdot \nabla (u_i-w_i)+g(u_i)(u_i-w_i)+ f(u_i) (u_i-w_i)\sum_{j\neq i}F(u_j))\, dx\\
										       &\geq&2 \sum_{i=1}^m \int_\Omega (\nabla u_i \cdot \nabla (u_i-w_i) +g(w_i)(u_i-w_i)+ f(w_i)(u_i-w_i)\sum_{j\neq i}F(u_j))\, dx\\
										       &=& 2 \sum_{i=1}^m \int_\Omega (\nabla u_i \cdot \nabla (u_i-w_i)-\nabla w_i\cdot \nabla (u_i-w_i) + \mu_i u_i(u_i-w_i))\, dx\\
										       &=& 2\langle u-w,u-w\rangle_{H^1_0(\Omega)}=2\|V(u)\|_{H^1_0(\Omega)}^2.
\end{eqnarray*}
\end{proof}

With $V$ we can now construct a $J$-decreasing flow for which $\Peh_\delta$ is positively invariant.

\begin{lemma}\label{lem:properties_of_pseudogradient}
There exists a unique global solution $\eta: \R^+\times \Mah \to H^1_0(\Omega; \R^m)$ for the initial value problem
\begin{equation}\label{eq:pseudogradient_flow}
\frac{d}{dt} \eta(t,u)=-V(\eta(t,u));\qquad \eta(0,u)=u\in \Mah.
\end{equation}
Moreover:
\begin{enumerate}
\item[$(i)$] $\eta(t,u)\in \Mah,\forall  t>0,\ u\in \Mah$;
\item[$(ii)$] For each $u\in \Mah$, the map $t\mapsto J(\eta(t,u))$ is nonincreasing.
\item[$(iii)$] There exists $\bar \delta$ such that, for every $\delta<\bar \delta$,
$$
\eta(t,u)\in \Peh_\delta \qquad \text{ whenever } u\in \Mah\cap \Peh_\delta, J(u)\leq \tilde d+1, \text{ and } t>0.
$$
\end{enumerate}
\end{lemma}

\begin{proof}
As $V\in C^1(\Mah^\ast)$, there exists a solution $\eta:[0,T_{max})\times \Mah^\ast \to H^1_0(\Omega;\R^m)$, where $T_{max}$ is the maximal time of existence of solution. We have
\begin{eqnarray*}
\frac{d}{dt} \int_\Omega \eta_i^2(t,u)\, dx &=& -2\int_\Omega \eta_i(t,u)V_i(\eta(t,u))\, dx\\
								&=&  2\int_\Omega \eta_i(t,u)K_i(\eta(t,u))\, dx-2\int_\Omega \eta^2_i(t,u)\, dx\\
								&=& 2-2\int_\Omega \eta^2_i(t,u)\, dx
\end{eqnarray*}
whence
$$
\frac{d}{dt} \Big( e^{2t}(\int_\Omega \eta_i^2(t,u)\, dx-1)\Bigr)=0.
$$
As $\int_\Omega \eta_i^2(0,u)\, dx=\int_\Omega u_i^2\, dx=1$, we get
$$
\int_\Omega \eta_i^2(t,u)\, dx=1 \qquad \text{ for every }t.
$$
Moreover, from this and Lemma \ref{lem:V_is_pseudogradient} we see that
$$
\frac{d}{dt} J(\eta(t,u))= - \langle \nabla J(\eta(t,u)),V(\eta(t,u))\rangle_{H^1_0(\Omega)}\leq -2 \| V(\eta(t,u))\|_{H^1_0(\Omega)}\leq 0.
$$
In particular, $\|\eta(t,u)\|^2_{H^1_0(\Omega)}\leq J(\eta(t,u))\leq J(u)<+\infty$ and thus $T_{max}=+\infty$ and $(i),(ii)$ hold.

(iii) Take $\bar \delta>0$ so that \eqref{eq:K_contracts_neigh_of_cone} holds for every $\delta<\bar \delta$. For every $u\in \Mah$ such that $J(u)\leq \tilde d+1$ and
$\dist_2(u,\Peh)=\delta<\bar \delta$, since
\begin{equation*}
\eta(t,u)= \eta(0,u)+t\dot \eta(0,u)+{\rm o}(t) = u-t V(u)+{\rm o}(t) \quad \text{ as } t\to 0,
\end{equation*}
we see that
\begin{eqnarray*}
\dist_2(\eta(t,u),\Peh)&=& \dist_2(u-t(u-K(u))+{\rm o}(t),\Peh)\\
			    &=&\dist_2((1-t)u+tK(u)+{\rm o}(t),\Peh)\\
			    &\leq& (1-t) \dist_2(u,\Peh)+t\dist_2(K(u),\Peh)+{\rm o}(t)\\
			    &<& (1-t)\delta + t\delta/2+{\rm o}(t)<\delta
\end{eqnarray*}
for sufficiently small $t>0$, and the conclusion follows.
\end{proof}

We can now conclude the section with a proof of the desired result.
\begin{proof}[Proof of Theorem \ref{thm:sign-changing_solution}]
Take any $\delta<\min\{\sqrt{2}/2,\bar \delta\}$ and denote $d_\delta^{k_1,\ldots,k_m}$ simply by $d$. In view of a contradiction, suppose there exists $0<\ep<1$ such that
\begin{equation}\label{eq:deformation_arg_contradiction}
\|V(u)\|^2_{H^1_0(\Omega)}\geq \ep,\qquad \forall u\in \Mah:\ |J(u)-d|\leq 2\ep,\quad \dist_2(u,\Peh)\geq \delta.
\end{equation}
Let us take any $A\in \Gamma^{(k_1,\ldots,k_m)}$ such that
$$
\sup_{A\setminus \Peh_\delta} J<d+\ep\leq \tilde d+1
$$
and consider $B:=\eta(1,A)$, where $\eta$ is defined by \eqref{eq:pseudogradient_flow}. From Lemmas \ref{lem:properties_of_vector_genus} and \ref{lem:A_setminus_Cone_is_nonempty} we know that $B\in \Gamma^{(k_1,\ldots,k_m)}$ and that $B\setminus \Peh_\delta\neq \emptyset$. Take $u\in A$ such that $\eta(1,u)\not\in \Peh_\delta$ and
$$
d-\ep\leq \sup_{B\setminus \Peh_\delta}J-\ep\leq J(\eta(1,u)).
$$
Since $\Peh_\delta$ is positively invariant for the flow $\eta$ (cf. Lemma \ref{lem:properties_of_pseudogradient}), we see that $\eta(t,u)\not\in \Peh_\delta$ for every $t\in [0,1]$. Moreover,
\begin{eqnarray*}
d&\leq& \sup_{B\setminus \Peh_\delta} J \leq J(\eta(1,u))+\ep \leq J(\eta(t,u))+\ep \leq J(u)+\ep\\
	   &\leq& \sup_{A\setminus \Peh_\delta} J +\ep < d+2\ep.
\end{eqnarray*}
We conclude from \eqref{eq:deformation_arg_contradiction} that $\| V(\eta(t,u))\|^2_{H^1_0(\Omega)}\geq  \ep$ for every $t\in [0,1]$ and
\begin{eqnarray*}
\frac{d}{dt} J(\eta(t,u)) &=& -\langle \nabla J(\eta(t,u)),V(\eta(t,u))\rangle_{H^1_0(\Omega)}\\
					     &\leq & -2 \| V(\eta(t,u))\|_{H^1_0(\Omega)}^2\leq -2\ep \qquad \forall t\in [0,1].
\end{eqnarray*}
Whence, after an integration,
$$
d-\ep\leq J(\eta(1,u))\leq J(u)-2\ep< d-\ep.
$$
Thus \eqref{eq:deformation_arg_contradiction} implies a contradiction and therefore we can find a sequence $u_n\in \Mah$ such that
$$
J(u_n) \to d, \quad V(u_n)\to 0 \quad \text{and}\quad \dist_2(u_n,P)\geq \delta.
$$
Lemma \ref{lem:palais_smale} now implies the existence of $u\in \Mah$ such that, up to a subsequence, $u_n\to u$ strongly in $H^1_0(\Omega;\R^m)$. Hence $J(u)=d$, $V(u)=0$, and $\dist_2(u,P)\geq \delta$, which yields the desired result.
\end{proof}

We have deduced that for each $f,g$ and $k_1,\ldots, k_m$ there exists $\delta=\delta(f,g,k_1,\ldots,k_m)<\min\{\sqrt{2}/2,\bar \delta\}$ such that $d_\delta^{k_1,\ldots, k_m}$ is a critical level for $J|_\Mah$. From now on we will denote such level simply by $d^{k_1,\ldots, k_m}$.

\subsection{Existence of infinitely many sign-changing solutions of \eqref{eq: BEC_more_general}}

In this subsection we will prove Theorem \ref{thm:infinitely_many_general}. For that, we will prove that $d^{k_1,\ldots,k_m}\to +\infty$ as some $k_i\to +\infty$.

\begin{lemma}\label{lemma:lower_bound}
Let $k_1,\ldots,k_m\geq 2$. Then
$$
d^{k_1,\ldots,k_m}\geq \sum_{i=1}^m \lambda_{k_i-1}(\Omega).
$$
\end{lemma}
\begin{proof}
Let $\{\varphi_k\}_k$ be the sequence of eigenfunctions of $(-\Delta, H^1_0(\Omega))$, normalized in $L^2(\Omega)$, associated to the eigenvalues $\{\lambda_k\}_k$. Take $A\in \Gamma^{(k_1,\ldots, k_m)}$ and consider, for each $i$, the function
\begin{eqnarray*}
&g_i:A\to \R^{k_i-1}&\\
&u\mapsto \Bigl(\int_\Omega u_i \varphi_1\, dx,\ldots, \int_\Omega u_i \varphi_{k_i-2}\, dx,\int_\Omega u_i |u_i|\, dx  \Bigr).&
\end{eqnarray*}
Then $g:=(g_1,\ldots,g_m)$ belongs to $F_{(k_1,\ldots,k_m)}(A)$ and there exists $\bar u\in A$ such that $\bar u_i\in \text{span}\{\varphi_1,\ldots,\varphi_{k_i-2}\}^\bot$ and $\bar u_i\not\in \Peh_\delta$, as $\delta<\sqrt{2}/2$. Thus
$$
\sup_{u\in A\setminus \Peh_\delta} J(u)\geq J(\bar u) \geq \sum_{i=1}^m \int_\Omega |\nabla \bar u_i|^2\, dx\geq \sum_{i=1}^m \lambda_{k_i-1}(\Omega),
$$
and the result follows.
\end{proof}

\begin{proof}[Proof of Theorem \ref{thm:infinitely_many_general}]
We have $\lambda_k(\Omega)\to +\infty$ as $k\to +\infty$, whence $d^{k_1,\ldots, k_m}\to +\infty$ as $k_i\to +\infty$ for some $i$ and the result follows.
\end{proof}

\section{Proof of Theorem \ref{thm:main1}}\label{sec:third_section}

Theorem \ref{thm:main1} is not an immediate consequence of Theorem \ref{thm:infinitely_many_general}. In fact, by choosing $f(t)=\sqrt{2\beta} t$ and $g(t)=a_i t^3$, we see that they do not satisfy condition $(fg2)$. Moreover, for $N\geq 5$ the associated energy functional
$$
J_\beta(u)=\sum_{i=1}^m \int_\Omega  (|\nabla u_i|^2 + \frac{a_i u_i^4}{2})\, dx + \mathop{\sum_{i,j=1}^m}_{i\neq j} \frac{\beta}{2}\int_\Omega  u_i^2 u_j^2 \, dx
$$
might take the value $+\infty$. We overcome these problems by considering suitable truncatures of the functions $t\mapsto t$ and $t\mapsto t^3$. Fix any $1<p<\min\{2^\ast/2, 3\}$ and $1<q<\min\{2^\ast,3\}$. Given $n\in \N$ we define the odd $C^1$ functions
$$
f_n(t)=\left\{
\begin{array}{cl}
t, & |t|\leq n\\[10pt]
\displaystyle \frac{t|t|^{p-2}}{(p-1)n^{p-2}} + n -\frac{n}{p-1}, & t\geq n\\[10pt]
\displaystyle \frac{t|t|^{p-2}}{(p-1)n^{p-2}} + \frac{n}{p-1}-n,  &t\leq -n
\end{array}
\right.
$$
$$
g_n(t)=\left\{
\begin{array}{cl}
t^3, & |t|\leq n\\[10pt]
\displaystyle \frac{3t|t|^{q-2}}{(q-1)n^{q-4}} + n^3 -\frac{3n^3}{q-1}, & t\geq n\\[10pt]
\displaystyle \frac{3t|t|^{q-2}}{(q-1)n^{q-4}} + \frac{3n^3}{q-1}-n^3,  &t\leq -n.
\end{array}
\right.
$$

and their primitives $F_n(t):=\int_0^t f_n(\xi)\, d\xi$, $G_n(t):=\int_0^t g_n(\xi)\, d\xi$.

\begin{lemma}\label{lemma:properties_of_f_n} The functions $f_n$ and $g_n$ satisfy the following properties:
\begin{itemize}
\item[$(i)$] for every $n\in \N$ there exists $C>0$ such that
$$
|f_n(t)|\leq C(1+|t|^{p-1}),\quad |g_n(t)|\leq C(1+|t|^{q-1}) \quad \text{ for every $t\in \R$};
$$
\item[$(ii)$] there exists $\theta>0$ independent of $n$ such that
\begin{equation}\label{eq:f_n-1}
f_n(t)t\leq \theta F_n(t) \quad \text{ and }\quad g_n(t)t\leq \theta G_n(t) \quad \text{ for every $t\in \R$};
\end{equation}
\item[$(iii)$] we have
\begin{equation*}\label{eq:f_n-2}
f_n(s)t+f_n(t)s\leq f_n(s)s+f_n(t)t\quad \text{ and } \quad g_n(s)t+g_n(t)s\leq g_n(s)s+g_n(t)t
\end{equation*}
for every $s,t\in \R$;
\item[$(iv)$] $g_n(t)\leq t^3$ for every $t\geq 0$.
\end{itemize}
\end{lemma}

\begin{proof}
(ii) Since both $f_n(t)t$ and $F_n(t)$ are even, it is enough to check \eqref{eq:f_n-1} for $t\geq 0$. Take $\theta':=\max\{1,p-1\}$, and let us show that
\begin{equation}\label{eq:f_n-auxiliary}
\Lambda(t):=\theta' f_n(t)-f_n'(t)t\geq 0 \qquad \forall t\geq 0.
\end{equation}
For $0\leq t\leq n$, $\Lambda(t)=\theta' t-t=(\theta'-1)t\geq 0$,
while for $t>n$
$$
\Lambda'(t)=\theta' \frac{t^{p-2}}{n^{p-2}}-(p-1)\frac{t^{p-2}}{n^{p-2}}\geq 0
$$
and hence, after an integration, $\Lambda(t)\geq \Lambda(n)\geq 0$.
Therefore \eqref{eq:f_n-auxiliary} holds and then
$$
f_n(t)t\leq (\theta'+1)F_n(t).
$$
The proof for $g_n$ is analogous, taking $\theta:=\theta'+1$ with $\theta':=\max\{3,q-1\}$.

(iii) Let us first consider the case $s,t\geq 0$. If $0\leq s,t\leq n$, we have
$$
f_n(s)t+f_n(t)s=2st\leq s^2+t^2=f_n(s)s+f_n(t)t.
$$
If $s,t\geq n$,
\begin{eqnarray*}
f_n(s)t+f_n(t)s &=&\Bigl(\frac{s^{p-1}}{(p-1)n^{p-2}}+n-\frac{n}{p-1}\Bigr)t + \Bigl(\frac{t^{p-1}}{(p-1)n^{p-2}}+n-\frac{n}{p-1}\Bigr)s\\
		       &=& \frac{s^{p-1}t+t^{p-1}s}{(p-1)n^{p-2}} + \Bigl(n-\frac{n}{p-1}\Bigr)t+ \Bigl(n-\frac{n}{p-1}\Bigr)s\\
		       &\leq& \frac{s^p+t^p}{(p-1)n^{p-2}}+ \Bigl(n-\frac{n}{p-1}\Bigr)t+ \Bigl(n-\frac{n}{p-1}\Bigr)s\\
		       &=& f_n(s)s+f_n(t)t.
\end{eqnarray*}
If $s\geq n$, $0\leq t\leq n$, since
$$
\Bigl( \frac{s^{p-1}}{(p-1)n^{p-2}}+n-\frac{n}{p-1}-t\Bigr)(s-t)\geq 0
$$
then we have
\begin{eqnarray*}
f_n(s)t+f_n(t)s &=&\Bigl(\frac{s^{p-1}}{(p-1)n^{p-2}}+n-\frac{n}{p-1}\Bigr)t+ts\\
			&\leq& \Bigl(\frac{s^{p-1}}{(p-1)n^{p-2}}+n-\frac{n}{p-1}\Bigr)s+t^2\\
			&=& f_n(s)s+f_n(t)t.
\end{eqnarray*}
Hence \eqref{eq:f_n-2} holds for $s,t\geq 0$. Finally, for $s\leq 0,t\geq 0$, we have
$$
f_n(s)t+f_n(t)s\leq 0\leq f_n(s)s+f_n(t)t
$$
while for $s,t\leq 0$
\begin{eqnarray*}
f_n(s)t+f_n(t)s   &=&   f_n(-s)(-t)+f_n(-t)(-s)\\
		         &\leq& f_n(-s)(-s)+f_n(-t)(-t)=f_n(s)s+f_n(t)t.
\end{eqnarray*}
The proof for $g_n$ is analogous.

(iv) We need to check that, for $t\geq n$,
$$
\Theta(t):=t^3-\frac{3t^{q-1}}{(q-1)n^{q-4}}-n^3+\frac{3n^3}{q-1}\geq 0.
$$
Now $\Theta(n)=0$, and
$$
\Theta'(t)=3t^2-\frac{3t^{q-2}}{n^{q-4}}\geq 0 \quad \text{ if and only if } \quad n^{q-4}\geq t^{q-4},
$$
which is true because $q<3$.
\end{proof}

Thus the truncated functions $f_n,g_n$ satisfy $(fg1)-(fg3)$, and hence from Theorem \ref{thm:infinitely_many_general} we immediately deduce, for each $n$, the existence of infinitely many sign-changing solutions of the problem
\begin{equation}\label{eq:BEC_general_with_beta}
-\Delta u_i + a_ig_n(u_i)+2\beta f_n(u_i) \sum_{j\neq i} F_n(u_j) =\lambda^n_{i,\beta} u_i, \qquad u_i \in H^1_0(\Omega).
\end{equation}

More precisely, if for each $n$ we define
$$
J_\beta^n(u)=\sum_{i=1}^m \int_\Omega (|\nabla u_i|^2 + 2a_iG_n(u_i))\, dx + 2\beta \mathop{\sum_{i,j=1}^m}_{j\neq i} \int_\Omega F_n(u_i) F_n(u_j)\, dx
$$
and the minimax levels
$$
c_{\beta,n}^{k_1,\ldots, k_m}=\inf_{A\in \Gamma^{(k_1,\ldots, k_m)}} \sup_{A\setminus \Peh_\delta} J_\beta^n,
$$
there exists an unbounded sequence $(u_\beta)_\beta$ of solutions of \eqref{eq:BEC_general_with_beta} such that $J_\beta^n(u_\beta)=c_{\beta,n}^{k_1,\ldots,k_m}$.

We can easily deduce an upper-bound for these minimax levels, independent of $n$. Indeed, consider the functional
$$
J_\infty(u)=\left\{
\begin{array}{cc}
\displaystyle \sum_{i=1}^m \int_\Omega (|\nabla u_i|^2 + \frac{a_i u_i^4}{2})\, dx & \text{ if } \begin{array}{c}u_i\cdot u_j\equiv0 \ \forall i\neq j,\\\text{ and } \int_\Omega u_i^4<\infty, \end{array}\\[10pt]
+\infty								& \text{otherwise.}
\end{array}
\right.
$$
Using Lemma \ref{lem:properties_of_vector_genus}-(i), one can construct $A\in \Gamma^{(k_1,\ldots,k_m)}$ such that $\sup_A J_\infty <\infty$ (see for instance the proof of Lemma \ref{lemma:relaxed_c_infty} ahead). Then we take
\begin{equation*}
c_\infty^{k_1,\ldots,k_m}:=\min_{A\in \Gamma^{(k_1,\ldots,k_m)}}\sup_{A} J_\infty<\infty.
\end{equation*}
As $2G_n(t)\leq t^4/2$ (\textit{cf.} Lemma \ref{lemma:properties_of_f_n}-(iv)) we have $J_\beta^n(u)\leq J_\infty(u)\ \forall u$, and hence
$$
c_{\beta,n}^{k_1,\ldots,k_m} \leq c_\infty^{k_1,\ldots,k_m} \qquad \text{ for all $\beta>0$ and $n\in \N$}.
$$

 Let
$$
\Kah_{\beta,n}^{k_1,\ldots,k_m}:=\Bigl\{u\in H^1_0(\Omega;\R^m):\ u \text{ satisfies } \eqref{eq:BEC_general_with_beta} \text{ and } J_\beta^n(u_\beta)=c_{\beta,n}^{k_1,\ldots,k_m}   \Bigr\}.
$$

By using a Brezis-Kato type argument, we have the following.

\begin{lemma}[a priori bounds]\label{lemma:brezis_kato} There exists a constant $C=C(\beta,k_1,\ldots,k_m)>0$, independent of $n$, such that
$$
\|u\|_{L^\infty(\Omega)}\leq C
\qquad \text{ for every } u\in \Kah_{\beta,n}^{k_1,\ldots,k_m}.
 $$
\end{lemma}

\begin{proof}
1. $(\lambda_{i,\beta}^n)_n$ are bounded, independently of $n$. We have
$$
J_{\beta}^n(u)=\sum_{i=1}^m \int_\Omega (|\nabla u_i|^2+2a_i G_n(u_i))\, dx + \sum_{j\neq i}\int_\Omega 2\beta F_n(u_i)F_n(u_j)\, dx =c_{\beta,n}^{k_1,\ldots,k_m}   \leq c_{\infty}^{k_1,\ldots,k_m},
$$
then
\begin{equation}\label{eq:bounds_independent_of_n}
\int_\Omega |\nabla u_i|^2\, dx,\ \int_\Omega a_i G_n(u_i)\, dx, \ \int_\Omega \beta  F_n(u_i)F_n(u_j)\, dx \leq c_{\infty}^{k_1,\ldots,k_m}.
\end{equation}
This, together with Lemma \ref{lemma:properties_of_f_n}-(ii) yields
$$
\int_\Omega a_i g_n(u_i)u_i\, dx \leq \theta \int_\Omega a_i G_n(u_i)\, dx\leq \theta a_i c_\infty^{k_1,\ldots,k_m},
$$
$$
\int_\Omega \beta f_n(u_i)u_i F_n(u_j)\, dx \leq \theta \int_\Omega \beta F_n(u_i)F_n(u_j)\, dx\leq \theta c_{\infty}^{k_1,\ldots,k_m}
$$
and hence
\begin{eqnarray*}
0\leq \lambda_{i,\beta}^n  &=&      \int_\Omega (|\nabla u_i|^2+ a_i g_n(u_i)u_i + 2\beta  f_n(u_i)u_i \sum_{j\neq i} F_n(u_j))\, dx\\
				         &\leq& (1+\theta a_i+ 2\beta(m-1)\theta)c_{\infty}^{k_1,\ldots,k_m}.
\end{eqnarray*}

2. Observe that from \eqref{eq:bounds_independent_of_n} we know there exists $C>0$ independent of $n$ such that
$$
\|u_i\|_{L^2(\Omega)}\leq C \qquad \forall u\in \Kah_{\beta,n}^{k_1,\ldots,k_m}.
$$
Suppose that $u\in L^{2+\delta}(\Omega;\R^m)$ for some $\delta$; we can test \eqref{eq:BEC_general_with_beta} with $u_i |u_i|^\delta$, obtaining
\begin{eqnarray*}
\frac{1+\delta}{(1+\frac{\delta}{2})^2}\int_\Omega |\nabla |u_i|^{1+\frac{\delta}{2}}|^2\, dx
	&\leq& \frac{1+\delta}{(1+\frac{\delta}{2})^2}\int_\Omega |\nabla |u_i|^{1+\frac{\delta}{2}}|^2\, dx +  \int_\Omega a_i g_n(u_i)u_i |u_i|^\delta\, dx\\
	& & +2\beta \int_\Omega f_n(u_i)u_i |u_i|^\delta \sum_{j\neq i} F_n(u_j)\, dx \\
	& = & \lambda_{i,\beta}^n \int_\Omega |u_i|^{2+\delta}\, dx
\end{eqnarray*}
Hence we have
\begin{eqnarray*}
\|u_i\|_{L^{2^\ast(2+\delta)/2}(\Omega)} &\leq& \Bigl( C_S^2 \frac{(1+\frac{\delta}{2})^2}{1+\delta} \Big)^\frac{1}{2+\delta} \Bigl(\lambda_{i,\beta}^n \int_\Omega |u_i|^{2+\delta} \, dx\Bigr)^\frac{1}{2+\delta}\\
						&\leq& \Bigl( C \frac{(1+\frac{\delta}{2})^2}{1+\delta} \Big)^\frac{1}{2+\delta} \|u_i\|_{L^{2+\delta}(\Omega)}.
\end{eqnarray*}
Now we iterate, by letting
$$
\delta(1)=0,\qquad 2+\delta(k+1)=2^\ast(2+\delta(k))/2.
$$
Observe that $\delta(k)\to \infty$, since $\delta(k)\geq (2^\ast/2)^{k-1}$.
We then have
\begin{eqnarray*}
\|u_i\|_{L^{2^\ast(2+\delta)/2}(\Omega)} &\leq& \prod_{j=1}^{k}\left[ C \frac{\left(1+\frac{\delta(j)}{2}\right)^2}{1+\delta(j)} \right]^{\frac{1}{2+\delta(j)}} \|u_i\|_{L^2(\Omega)}\\
			&\leq& \exp\left(\sum_{j=1}^\infty \frac{1}{2+\delta(j)}\log\left[ C \frac{\left(1+\frac{\delta(j)}{2}\right)^2}{1+\delta(j)} \right]   \right)\|u_i\|_{L^2(\Omega)}
\end{eqnarray*}
As $\delta(j)\geq (2^\ast/2)^{j-1}$, we see that
\[
\sum_{j=1}^\infty \frac{1}{2+\delta(j)}\log\left[ C \frac{\left(1+\frac{\delta(j)}{2}\right)^2}{1+\delta(j)} \right] < \infty,
\]
which provides the uniform bound in $L^\infty(\Omega)$
\end{proof}

\begin{proof}[Proof of Theorem \ref{thm:main1}]
Let $C>0$ be the constant appearing in the previous lemma and take $n\geq C$. Then there exist infinitely many sign-changing solutions $u$ of \eqref{eq:BEC_general_with_beta}. By the choice of $n$, each solution of \eqref{eq:BEC_general_with_beta} is also a solution of \eqref{eq:Bose-Einstein}, and the result follows.
\end{proof}


\section{Optimal partition problems. Proof of Theorem \ref{thm:main2}}\label{sec:fourth_section}

The purpose of this section is to prove Theorem \ref{thm:main2}. Hence we consider $a_1,\ldots,a_m=0$, and we are dealing with system \eqref{eq:Bose-Einstein_omega_i=0}. Before concentrating our attention on the case of optimal partition problems involving the second eigenvalue, let us prove some preliminary statements.

\begin{lemma}\label{lemma:relaxed_c_infty}
Let $k_1,\ldots,k_m\in \N$. We have
$$
\inf_{(\omega_1,\ldots,\omega_m)\in \Peh_m(\Omega)} \sum_{i=1}^m \lambda_{k_i}(\omega_i)\geq c_\infty^{k_1,\ldots,k_m},
$$
where
$$
\Peh_m(\Omega)=\Bigl\{ (\omega_1,\ldots,\omega_m):\ \omega_i\subseteq \Omega \text{ are open sets, and } \omega_i\cap \omega_j=\emptyset,\ \forall i\neq j\Bigr\}.
$$
\end{lemma}
\begin{proof}
Take $(\omega_1,\ldots,\omega_m)\in \Peh_m(\Omega)$ and, for each $i\in \{1,\ldots,m\}$, let $\vphi_1^i,\ldots, \vphi_{k_i}^i$ denote the first $k_i$ eigenfunctions of $(-\Delta, H^1_0(\omega_i))$, normalized in $L^2(\Omega)$. Let
$$
A_i:= \{u\in \text{span}\{ \vphi^i_1,\ldots, \vphi^i_{k_i} \}:\ \|u\|_{L^2(\Omega)}=1\}.
$$
Then there exists an obvious odd homeomorphism between $A_i$ and $S^{k_i-1}$, the unitary sphere in $\R^{k_i}$. Therefore Lemma \ref{lem:properties_of_vector_genus}-(i) applies, yielding that
$$
\prod_{i=1}^m A_i \in \Gamma^{(k_1,\ldots,k_m)}.
$$
As $\omega_i\cap \omega_j=\emptyset$ for $i\neq j$, it is now easy to conclude that
\begin{eqnarray*}
\sum_{i=1}^m \lambda_{k_i}(\omega_i)=\sum_{i=1}^m \max_{u_i\in A_i} \int_\Omega |\nabla u_i|^2\, dx &=&\max_{u\in \prod_{i=1}^m A_i} \sum_{i=1}^m \int_\Omega |\nabla u_i|^2\, dx\\
			&\geq& \inf_{A\in \Gamma^{(k_1,\ldots, k_m)}} \sup_A J_\infty(u)=c_\infty^{k_1,\ldots, k_m}.
\end{eqnarray*}
\end{proof}

\begin{rem}
By combining Lemmas  \ref{lemma:lower_bound} and \ref{lemma:relaxed_c_infty} we know that for each $k_1,\ldots,k_m\geq 2$ there exists a minimax level $c_\beta^{k_1,\ldots,k_m}$ associated with \eqref{eq:Bose-Einstein_omega_i=0} such that
$$
\sum_{i=1}^m \lambda_{k_i-1}(\Omega)\leq c^{k_1,\ldots,k_m}_\beta\leq \inf_{(\omega_1,\ldots,\omega_m)\in \Peh_m(\Omega)} \sum_{i=1}^m \lambda_{k_i}(\omega_i).
$$
Since $\lambda_k \simeq k^{2/N}$, then there exist two constants $C_1,C_2>0$, independent of $k_i$, such that
$$
C_1 \sum_{i=1}^m k_i^{2/N}\leq c^{k_1,\ldots,k_m}_\beta \leq C_2 \sum_{i=1}^m k_i^{2/N}.
$$
These inequalities can be used to estimate how many critical levels there are in each interval $[a,b]\subseteq \R^+$.
\end{rem}

By what we have seen in the previous section, we know that for each $k_1,\ldots, k_m\geq 2$, $\beta>0$, there exists $u_\beta=(u_{1,\beta},\ldots,u_{m,\beta})$, a sign-changing solution of \eqref{eq:Bose-Einstein}, satisfying $J_\beta(u_\beta)\leq c_\infty^{k_1,\ldots, k_m}$. By combining the results in the works \cite{nttv1, tavaresterracini}, we have the following informations about the asymptotic behavior of the solutions $u_\beta$ as $\beta\to +\infty$.

\begin{thm}\label{thm:conclusions_limiting_problem}
There exists a vector Lipschitz function $\overline u=(\bar u_1,\ldots,\bar u_m)\in \Mah$ such that, up to a subsequence,
\begin{enumerate}
\item[$(i)$] $ u_{i,\beta}\to \bar u_i$ in $H^1_0(\Omega)\cap C^{0,\alpha}(\overline \Omega)$ for every $0<\alpha<1$;
\item[$(ii)$] $-\Delta u_i=\lambda_i u_i$ in the open set $\{u_i\neq 0\}$, where $\lambda_i:=\lim_\beta \lambda_{i,\beta}$;
\item[$(iii)$] $\bar u_i\cdot \bar u_j\equiv 0$ and $\displaystyle \int_\Omega \beta u_{i,\beta}^2 u_{j,\beta}^2\to 0$ as $\beta\to +\infty$, whenever $i\neq j$;
\item[$(iv)$] the nodal set $\Gamma_{\bar U}:=\{x\in \Omega:\ \bar u_i(x)=0\}$ consists, up to a set having at most Hausdorff dimension $N-2$, of a union of hypersurfaces of class $C^\infty$.
\end{enumerate}
\end{thm}
\begin{proof}
Let us show first of all that the $\lambda_{i,\beta}$'s appearing in \eqref{eq:Bose-Einstein} are uniformly bounded in $\beta$, and that there exists $C>0$, independent of $\beta$, such that
\begin{equation}\label{eq:conclusion_aux1}
\|u_\beta\|_{L^\infty(\Omega)}\leq C.
\end{equation}
In fact,
$$
J_\beta(u_\beta)=\sum_{i=1}^m \int_\Omega |\nabla u_{i,\beta}|^2\, dx + \mathop{\sum_{i,j=1}^m}_{j\neq i}\frac{\beta}{2}\int_\Omega  u_{i,\beta}^2 u_{j,\beta}^2\, dx \leq C
$$
and hence
$$
\int_\Omega |\nabla u_{i,\beta}|^2\, dx,\ \int_\Omega \beta u_{i,\beta}^2 u_{j,\beta}^2\, dx \leq C\quad \forall i,\ \beta>0.
$$
Therefore $\lambda_{i,\beta}=\int_\Omega (|\nabla u_{i,\beta}|^2 + \beta u_{i,\beta}^2 \sum_{j\neq i}  u_{j,\beta}^2 )\, dx$ is bounded independently of $\beta$. Moreover, $u_\beta$ is a bounded sequence in $H^1_0(\Omega;\R^m)$, thus a Brezis-Kato type argument as the one shown in the proof of Lemma \ref{lemma:brezis_kato} gives \eqref{eq:conclusion_aux1}. Now $(i)$-$(iii)$ follows from \cite[Theorems 1.1 $\&$ 1.2]{nttv1} (see also Remark 3.11) and $(iv)$ is a direct consequence of \cite[Theorem 1.1]{tavaresterracini}. It should be stressed that although in \cite{nttv1} the results are stated for non-negative solutions, they also hold for solutions with no sign-restrictions; all arguments there can be adapted with little extra effort to this more general case, working with the positive and negative parts of a solution.
\end{proof}

Coming to the proof of Theorem \ref{thm:main2}, let us fix from now on $k_1=\ldots=k_m=2$. The importance of having obtained sign-changing solutions is clarified in the following key result.

\begin{lemma}\label{lemma:second_eigenvalue}
Within the notations of the previous theorem, for every $i\in \{1,\ldots,m\}$ we have that
$$
\int_\Omega |\nabla \bar u_i|^2\, dx\geq \lambda_2(\{\bar u_i\neq 0\}).
$$
\end{lemma}
\begin{proof}
Observe that $u_{i,\beta}$ is an eigenfunction of the operator $-\Delta +\beta \sum_{j\neq i} u_{j,\beta}^2$ in $H^1_0(\Omega)$ with eigenvalue $\lambda_{i,\beta}$. Since $u_{i,\beta}$ is a sign-changing solution, we have that
$$
\lambda_{i,\beta}>\lambda_1(-\Delta+\beta \sum_{j\neq i} u_{j,\beta}^2,\Omega),
$$
the first eigenvalue of $-\Delta+\beta \sum_{j\neq i} u_{j,\beta}^2$ in $H^1_0(\Omega)$. Moreover, if $\vphi_{i,\beta}\geq 0$ is such that $\|\vphi_{i,\beta}\|_{L^2(\Omega)}=1$ and
$$
-\Delta \vphi_{i,\beta}+\beta \vphi_{i,\beta}  \sum_{j\neq i}u_{j,\beta}^2=\lambda_1(-\Delta+\beta \sum_{j\neq i} u_{j,\beta}^2,\Omega) \vphi_{i,\beta},
$$
then
$$
\int_\Omega u_{i,\beta} \vphi_{i,\beta}\, dx=0.
$$
As
$$
\int_\Omega ( |\nabla \vphi_{i,\beta}|^2+\beta \vphi_{i,\beta}^2\sum_{j\neq i}u_{j,\beta}^2)\, dx = \lambda_1(-\Delta+\beta \sum_{j\neq i} u_{j,\beta}^2,\Omega)<\lambda_{i,\beta}\leq C,
$$
there exists $\bar \vphi_i\geq 0$ with $\|\bar \vphi_i\|_{L^2(\Omega)}=1$ such that $\vphi_{i,\beta}\rightharpoonup \bar \vphi_i$ weakly in $H^1_0(\Omega)$ and moreover by Fatou's Lemma
$$
\int_\Omega \bar \vphi_i^2\sum_{j\neq i}\bar u_j^2\, dx \leq \liminf_{\beta\to +\infty} \int_\Omega \vphi_{i,\beta}^2 \sum_{j\neq i} u_{j,\beta}^2\, dx \leq \lim_{\beta\to+\infty }C/\beta = 0,
$$
hence $\vphi_i=0$ a.e. on $\cup_{j\neq i} \{\bar u_j\neq 0\}$. Since $\Gamma$ has zero Lebesgue measure (recall Theorem \ref{thm:conclusions_limiting_problem}-(iv)) and $\bar \vphi\not\equiv 0$ in $\Omega$, then $\bar \vphi_i\not\equiv 0$ a.e. over $\{\bar u_i\neq 0\}$ and, as
$$
\int_\Omega \bar u_i \bar \vphi_i\, dx=0,
$$
then $\bar u_i\in H^1_0(\{\bar u_i\neq 0\})$ is a sign changing solution of $-\Delta \bar u_i=\lambda_i \bar u_i$, and
$$
\int_\Omega |\nabla \bar u_i|^2\, dx =\lambda_i\geq \lambda_2(\{\bar u_i\neq 0\}).
$$
\end{proof}

Now we are in a position to prove our second main result.

\begin{proof}[Proof of Theorem \ref{thm:main2}]
By combining everything we have done so far, we obtain
\begin{eqnarray*}
\inf_{(\omega_1,\ldots,\omega_m)\in \Peh_m(\Omega)} \sum_{i=1}^m \lambda_2(\omega_i)\geq c_\infty^{2,\ldots,2}\geq J_\beta(u_\beta)= \sum_{i=1}^m \int_\Omega  |\nabla u_{i,\beta}|^2\, dx + \mathop{\sum_{i,j=1}^m}_{j\neq i} \frac{\beta }{2} \int_\Omega  u_{i,\beta}^2 u_{j,\beta}^2\, dx
\end{eqnarray*}
and hence, by Theorem \ref{thm:conclusions_limiting_problem}-(i),(iii) and Lemma \ref{lemma:second_eigenvalue},
\begin{eqnarray*}
\inf_{(\omega_1,\ldots,\omega_m)\in \Peh_m(\Omega)} \sum_{i=1}^m \lambda_2(\omega_i)&\geq& c_\infty^{2,\ldots,2} \geq \lim_{\beta\to +\infty} J_\beta(u_\beta)=\sum_{i=1}^m \int_\Omega |\nabla \bar u_i|^2\\
							&\geq& \sum_{i=1}^m \lambda_2(\{\bar u_i\neq 0\}) \geq \min_{(\omega_1,\ldots,\omega_m)\in \Peh_m(\Omega)} \sum_{i=1}^m \lambda_2(\omega_i).
\end{eqnarray*}
Thus $(\{\bar u_1\neq 0\},\ldots,\{\bar u_m\neq 0\})$ is a solution of \eqref{eq:optimalpartition}, and the result now follows from Theorem \ref{thm:conclusions_limiting_problem}-(iv).
\end{proof}

\subsection{Further extensions}
If $k_1,\ldots,k_m=1$, observe that we can use similar (even easier) arguments without using the cones $\Peh_i$. In this case,
$$
c_\beta^{1,\ldots, 1}:= \inf_{A\in \Gamma^{1,\ldots,1}} \sup_A J_\beta=\inf_\Mah J_\beta
$$
and we can prove the existence of $u_\beta$, solution of \eqref{eq:Bose-Einstein}, such that $J_\beta(u)=c_\beta^{1,\ldots,1}$. Thus
$$
\inf_{(\omega_1,\ldots,\omega_m)\in \Peh_m(\Omega)} \sum_{i=1}^m \lambda_{1}(\omega_i)\geq c_\infty^{1,\ldots,1}\geq \lim_{\beta\to +\infty} J_\beta(u_\beta)=\sum_{i=1}^m \int_\Omega |\nabla \bar u_i|^2\, dx\geq \sum_{i=1}^m \lambda_1(\{\bar u_i\neq 0\})
$$
and hence we get the same result as in Theorem \ref{thm:main2} with \eqref{eq:optimalpartition} replaced by
$$
\inf_{(\omega_1,\ldots,\omega_m)\in \Peh_m(\Omega)} \sum_{i=1}^m \lambda_1(\omega_i).
$$
We recover this way the result already shown in \cite{CL_eigenvalues}.

More generally, we can also replace \eqref{eq:optimalpartition} by the problem
$$
\inf_{(\omega_1,\ldots,\omega_m)\in \Peh_m(\Omega)} \Bigl( \sum_{i=1}^{\overline m} \lambda_1(\omega_i)+\sum_{i=\bar{m}+1}^m \lambda_2(\omega_i)  \Bigr),\qquad \text{ where $\overline m<m$},
$$
getting the same conclusions as before. In fact, the same arguments of Sections \ref{sec:second_section}, \ref{sec:third_section} can be applied, with the difference that we just take
$$
\Peh:=\cup_{i=\overline m+1}^m (\Peh_i\cup -\Peh_i)
$$
in the definition of \eqref{eq:general_minmax_level}.\\

We conjecture that, given arbitrary $k_1,\ldots, k_m\in \N$, there exists $u_\beta=(u_{1,\beta}, \ldots, u_{m,\beta})$ solution of \eqref{eq:Bose-Einstein} with $J_\beta(u_\beta)\leq c_\infty^{k_1,\ldots, k_m}$, and $(\bar u_1,\ldots, \bar u_m)$ a limiting profile in the sense of Theorem \ref{thm:conclusions_limiting_problem}, such that $(\omega_1,\ldots, \omega_m):=(\{\bar u_1\neq 0\},\ldots, \{\bar u_m\neq 0\})$ solves
$$
\inf_{(\omega_1,\ldots, \omega_m)\in \Peh_m(\Omega)} \sum_{i=1}^m \lambda_{k_i}(\omega_i).
$$

{\bf Acknowledgments.}

H. Tavares was supported by FCT, grant SFRH/BPD/69314/2010 and Financiamento Base 2008 - ISFL/1/209.\\


\end{document}